\documentclass[12pt]{amsart}

\usepackage{amscd,amssymb,amsfonts,amsmath,amstext,amsthm}

\newcommand{\eps}{\varepsilon}

\newcommand{\La}{\Lambda}

\newcommand{\ds}{\displaystyle}

\def\BN{{\mathbb N}}

\def\BZ{{\mathbb Z}}
\def\BC{{\mathbb C}}
\def\CO{{\mathcal O}}

\def\D{{\mathcal D}}

\def\iff{\Leftrightarrow}

\newcommand{\cyc}[1]{\langle #1 \rangle}
\newcommand{\floor}[1]{\left\lfloor #1 \right\rfloor}
\newcommand{\ceil}[1]{\left\lceil #1 \right\rceil}
\newcommand{\ip}[2]{\langle#1,#2\rangle}

\DeclareMathOperator{\class}{class}

\newtheorem{thm}{Theorem}[section]
\newtheorem{df}[thm]{Definition}

\newtheorem{cor}[thm]{Corollary}
\newtheorem{prop}[thm]{Proposition}
\newtheorem{lem}[thm]{Lemma}

\numberwithin{equation}{thm}

\def\remark{\noindent\textbf{Remark:} \ }

\def\G{\BZ_{2^l}\rtimes D_k}

\title[Higher Indicators for the Doubles of some Totally Orthogonal Groups]
{Higher Indicators for the Doubles of some Totally Orthogonal Groups}
\author{Marc Keilberg}
\email{mkeilberg@ucsd.edu}

\begin{document}
\begin{abstract}
We investigate the indicators for certain groups of the form $\BZ_k\rtimes D_l$ and their doubles, where $D_l$ is the dihedral group of order $2l$.  We subsequently obtain an infinite family of totally orthogonal, completely real groups which are generated by involutions, and whose doubles admit modules with second indicator of -1.  This provides us with answers to several questions concerning the doubles of totally orthogonal finite groups.
\end{abstract}

\maketitle

\section*{Introduction}
Indicators of modules over Hopf algebras \cite{FGSV}\cite{LM}, especially their higher analogues \cite{KSZ2} considered herein, are proving to be a very useful invariant in the study of Hopf algebras.  For example, they have been used in classifying Hopf algebras themselves \cite{Ka}\cite{NS1}; in studying possible dimensions of the representations of $H$ \cite{KSZ1}; and in studying the prime divisors of the exponent of $H$ \cite{KSZ2}\cite{NS2}. Moreover, the indicator is invariant under equivalence of monoidal categories \cite{MaN}, and thus yield a gauge invariant of the category of finite dimensional representations. Another motivation comes from conformal field theory; see the work of Bantay \cite{B1}\cite{B2}.  The notion of higher indicators has also been extended to more general categories \cite{NS2}\cite{NS3}\cite{NS1}, where quasi-Hopf algebras play an important role \cite{N1}\cite{N2}.

It is well-known that the higher indicators for modules over groups are integers.  It was hoped the same would be true for the Drinfel'd double of a group.  The author proved this to be true for an infinite family of groups which includes the dihedral groups\cite{K}, and in \cite{Co} it was proven for symmetric groups acting on at most 10 elements.  More recently, \cite{IMM} showed it was true for a number of other groups, including every symmetric group, but also provided an example where the indicators are not integers.  The method used is nonconstructive, however, in that it only gives a sufficient condition to predict integrality, and provides no information on the specific values assumed beyond that.  Indeed, what connection there is between the structure of the group and the indicators of its double--which can be defined entirely in terms of the group, its subgroups, and their representations--remains a major question. For example, it is known that the second indicators determine when a module is self-dual, and if so if there is a symmetric or skew-symmetric invariant bilinear form \cite{LM}, but there is no known geometric interpretation for the higher indicators.

Determining if the indicators can take negative values or not remains a very difficult open question, even when considering just the group algebras \cite{Co}\cite{K}\cite{Sch}.  As such, it is crucial to pursue further examples in this direction.  In \cite{K} several examples were explicitly produced and calculated which had negative indicators in the double but none for the group.  Further similar examples will be computed in this paper, this time involving totally orthogonal groups.  Whether or not negative indicators exist for the doubles of $S_n$ for $n>10$ ($n\leq 10$ is known to be non-negative by \cite{Co}), or any of the classical reflection groups, remains an important open problem.

The groups we investigate are defined in Definition \ref{presentation}, and our results about them are summarized in Theorem \ref{doublesum}.  The procedure, notation, and organization are virtually identical to those in \cite{K}.  We work over the field $k=\BC$.  We will use $\wedge$ to denote logical conjunction, $\vee$ for logical disjunction, and $\oplus$ for exclusive disjunction in logical propositions.  Given a logical proposition $P$, we define $\delta_P$ to be 1 when $P$ is true, and otherwise 0.


\section{Preliminaries}\label{prelimsect}
We cite here only the bare essentials.  More comprehensive preliminary details can be found in \cite{K}.
\begin{df}\label{indicdef}
Given a simple $H$-module $V$ and its character $\chi$, we define the functions
$$\La^{[m]} = \sum \Lambda _1\Lambda _2\cdots\Lambda_m$$
$$\nu_m (\chi)=\chi \left( \La^{[m]}\right), m\in\BN$$
where  $\Lambda$ is the unique integral in $H$ with $\eps(\Lambda) = 1$.
\end{df}

Throughout, the indicator of a simple module is just the indicator of its irreducible
character.  When convenient, we will also use the representation corresponding to the associated character when denoting indicators of simple modules.

For any finite dimensional Hopf algebra $H$, there is an associated Hopf algebra $\D(H)$ known as the Drinfel'd double of $H$.  We give a cursory definition here that applies when $H$ is group algebra, and refer to \cite{Mo} for a more general and complete definition.

\begin{df}\label{doubledef}
  Let $G$ be a finite group.  We define a new hopf algebra $\D(G)$, the Drinfel'd double of $G$, in the following fashion.  As a vector space, $\D(G) = (G^{k})^{cop}\otimes_k kG$, where $G^k$ is the linear dual of the group algebra $kG$.  As a coalgebra, $\D(G)$ is just the usual tensor product of coalgebras.  To get the algebra structure, we use the $G$-action on $(G^{k})^{cop}$ coming from conjugation.
  
  As in \cite{Mo} we will write $h\bowtie f$, $h\in G, f\in (G^k)^{cop}$, for a typical element of $\D(G)$
\end{df}

We will need an explicit description for the irreducible modules of $\D(G)$ when $G$ is a finite group.  Our notation is the same as that used in \cite{AF}.

\begin{prop}\label{MOrho}\cite{DPR}\cite{Ma1}
    Let $G$ be a finite group.  The isomorphism classes of the irreducible $\D(G)$-modules are the modules $M(\CO,\rho)$ (defined below), where $\CO=\class(s)$ is the conjugacy class of some $s\in G$ and $\rho$ is (an isomorphism class of) an irreducible representation of $C_G(s)$ on a vector space $V$.  If we enumerate $\class(s) = \{t_1, ..., t_N\}$, where by convention we take $t_1=s$, and fix $g_i\in G$ with $t_1^{g_i}=t_i$ for $1\leq i\leq N$, then we can describe the module $M(\CO,\rho)$ in the following fashion:

    As a vector space $M(\CO,\rho) = \bigoplus_{i=1}^N g_i\otimes V$, or $N$ copies of $V$ indexed by the $g_i$.  We denote an element $g_i\otimes v, v\in V$ by $g_i v$.  For the left G-module structure, we define $g.g_iv = g_j(\gamma v)$, where $g_j$ and $\gamma$ are the (necessarily unique) elements with $g g_i= g_j\gamma$ in $G$ and $g_j\in\CO, \gamma\in C_G(s)$.  Here, $\gamma$ acts on $V$ via $\rho$.  For the left $G^{*}$ structure, we specify an equivalent left $G$-comodule structure $\delta$.  Specifically, we take $\delta(g_iv) = t_i\otimes g_iv$.  In particular, $M(\CO,\rho)$ can be graded by the elements of $\CO$.
\end{prop}

The specific choices of $s$ (the representative of $\CO$), the isomorphism class representative $\rho$ and corresponding vector space $V$, the enumeration of $\CO$ and the choice of the $g_i$ are not crucial.  They will all yield isomorphic $\D(G)$-modules.  In particular, we are free in the subsequent to fix these choices as suits the situation, and we will do so without further comment.

In order to study the indicators of $G$ and $\D(G)$, for $G$ a finite group, we need to introduce a few more pieces of notation.

\begin{df}\label{Gmsetdef}
Let $G$ be a finite group.  For any $x,y\in G$ and $m\in \BN$, define
\begin{eqnarray*}
    G_m(x) &=& \left\{ a\in G : \prod_{j=0}^{m-1} a^{-j}xa^j = 1\right\}\\
    G_m(x,y) &=& \left\{ a\in G : \prod_{j=0}^{m-1} a^{-j}x a^j = 1 \mbox{ and } x^m=y\right\}\\
    z_m(x,y) &=& |G_m(x,y)|.
\end{eqnarray*}
\end{df}
In the notation of \cite{KSZ2}, taking $F=G$ our $G_m(x,y)$ and $z_m(x,y)$ are precisely $G_{m,1}(x,y)$ and $z_{m,1}(x,y)$, respectively.

\begin{thm}\label{sweedpow}\cite{KSZ2}
Let $G$ be a finite group and let $\La$ be the integral of $\D(G)$.  Then
\begin{eqnarray*}
    \La^{[m]} &=& \frac{1}{|G|} \sum_{g,y\in G} z_m(g,y) p_g\bowtie y\\
    &=& \frac{1}{|G|} \sum_{g\in G, \ a\in G_m(g) } p_g\bowtie a^m
\end{eqnarray*}
\end{thm}

\begin{cor}\label{indicforms}\label{cor34} \cite{KSZ2}
Let $V=M(\class(s),\rho)$ be an irreducible $\D(G)$-module, as defined in Proposition \ref{MOrho}.  In particular, assume that $\rho$ is an irreducible representation of $C_G(s)$.  Let $\chi$ be the character of $V$, and $\eta$ the character of $\rho$.  Then
    \begin{eqnarray*}
      \nu_m(\chi) &=& \frac{1}{|C_G(s)|} \sum_{y\in C_G(s)} z_m(s,y)\eta(y)\\
      &=& \frac{1}{|G|} \sum_{g\in G, \ a\in G_m(g) } \chi(p_g\bowtie a^m)\\
      &=& \nu_m(\chi) = \frac{1}{|G|}\sum_{g\in\CO_s, \ a\in G_m(g)} \chi(p_g\bowtie a^m)
    \end{eqnarray*}
\end{cor}

Although the first equality is much more compact, for our purposes we will find the third a little bit easier to calculate with, since we can separate the two conditions $a\in G_m(g)$ and $a^m=y$ in the definition of $G_m(g,y)$.

Computing the indicators of modules associated to central elements is readily achieved by the following result.

\begin{lem} \label{FxF}\cite{K} Let $G$ be a finite group and suppose $x\in Z(G)$.  Let $V=M(x,\rho)$ be any corresponding irreducible $\D(G)$-module, as given in Proposition \ref{MOrho}.  Let $\chi$ be the character of $V$ and $\eta$ the character of $\rho$.
\begin{enumerate}
    \item Assume that $x^m=1$. Then the value $\nu_m(\chi)$ is
    exactly the same as the value of $\nu_m$ for $\rho$ (a $G$-module).  As a slight abuse of notation, we write $\nu_m(\chi) = \nu_m(\eta)$.
    \item Assume that $x^m\neq 1$. Then $\nu_m(\chi)=0$
\end{enumerate}
\end{lem}

\section{Groups of the form $\G$}\label{grpdefsect}

Our goal for the rest of the paper will be to analyze the indicators and other properties of certain totally orthogonal finite groups and their doubles.

\begin{df}\label{presentation}
Let $k,l\in\BN$ with $l\geq 3$ and $4\mid k$.  Let $D_k$ denote the dihedral group of order $2k$, and set $n_1 = 2^{l-1}+1$, $n_2=2^{l-1}-1$.  For any such pair, we consider the groups $\BZ_{2^l}\rtimes D_k$ given by the presentation
\begin{eqnarray*}
  \BZ_{2^l}\rtimes D_k = \cyc{a,u,v \ | \ a^{2^l}=u^k=v^2=1, uau^{-1}=a^{n_1}, vav = a^{n_2}, vuv=u^{-1}}.
\end{eqnarray*}
\end{df}
\remark There are other possible actions of $D_k$ on $\BZ_{2^l}$.  Our particular choice of action, as well as the particular restriction on $l$ and the order of the cyclic group, are a matter of convenience to obtain the desired examples.  The precise structure of a more generic group $\BZ_l\rtimes D_k$ varies significantly with different choices.  We will see later that the particular groups we consider are all totally orthogonal, amongst other things, and furthermore their doubles have indicators which are negative (see Theorem \ref{doublesum}).  The particular example of $\BZ_8\rtimes D_4$, with presentation as above, was previously known to D. Naidu (private communication).\newline

\indent We start by stating some of the key identities that hold in such a group.  These are easy consequences of the definition.
\begin{cor}\label{basicident}
Let $G=\G$ and $n_1,n_2$ be as in Definition \ref{presentation}.  Then we have the following.
\begin{enumerate}
  \item $n_1^2\equiv n_2^2\equiv 1\bmod 2^l$
  \item $n_1\equiv -n_2\bmod 2^l$.  In particular, $n_1 n_2\equiv -1\bmod 2^l$.
  \item $a^{n_1 i}=a^i \iff 2\mid i$
  \item $a^{n_1 i}=a^{i+2^{l-1}} \iff 2\nmid i$
  \item $a^{n_2 i}=a^{-i} \iff 2\mid i$
  \item $a^{n_2 i}=a^{2^{l-1}-i} \iff 2\nmid i$
  \item $$(a^s u^i v)^2 = \left\{ \begin{array}{cll}
    a^{2^{l-1}} &;& 2\mid i \wedge 2\nmid s\\
    1 &;& 2\nmid i \vee 2\mid s \end{array}\right.$$
\end{enumerate}
\end{cor}

\begin{cor}\label{invgen}
    Let $\G$ be as in Definition \ref{presentation}.  Then $\G$ is generated by involutions.  In particular, $\G=\cyc{v,uv,auv}$.
\end{cor}
\begin{proof}
  That $\G=\cyc{v,uv,auv}$ is immediate.  We easily compute that
  $$uv\cdot uv = u u^{-1} v v = 1$$
  $$auv\cdot auv = a a^{-1}uv uv = 1$$
  and the result follows.
\end{proof}

For the rest of this section, we continue to investigate the structure of $\G$.  We will need to know the conjugacy classes and centralizers in particular by Proposition \ref{MOrho}.

\begin{lem}\label{center}
    Let $G=\G$ be as in Definition \ref{presentation}.  Then
    $$Z(G)=\cyc{a^{2^{l-1}},u^{k/2}}\cong \BZ_2\times\BZ_2.$$
\end{lem}
\begin{proof}
  Suppose that $a^i u^j v^x\in Z(G)$.  Then
  \begin{eqnarray}
    va^i u^j v^x v &=& a^{n_2 i} u^{-j} v^x\label{vconj}\\
    u a^i u^j v^x u^{-1} &=& a^{n_1 i}u^{j+1-(-1)^{x}}v^x\label{uconj}\\
    a a^i u^j v^x a^{-1} &=& a^{i+1-n_1^j n_2^x}u^j v^x\label{aconj}
  \end{eqnarray}
  We then have that
  \begin{eqnarray*}
    (\ref{vconj}) = a^i u^j v^x &\iff& 2^{l-1}\mid i \wedge k\mid 2j\\
    (\ref{uconj}) = a^i u^j v^x &\iff& 2\mid i \wedge 2\mid x\\
    (\ref{aconj}) = a^i u^j v^x &\iff& n_1^j n_2^x\equiv 1\bmod 2^l \iff 2\mid j \wedge 2\mid x.
  \end{eqnarray*}
  Combining and using the facts that $4\mid k$ and $G=\cyc{a,u,v}$, we obtain that $Z(G) = \cyc{a^{2^{l-1}},u^{k/2}}$.  The isomorphism is then obvious.
\end{proof}

\begin{prop}\label{classes}
    Let $G=\G$ and $n_1,n_2$ be as in Definition \ref{presentation}.  The conjugacy classes of $G$ are given as follows.
    \begin{enumerate}
      \item $$\class(a^s u^{2i}) = \{ a^s u^{2i}, a^{n_1 s}u^{2i}, a^{n_2 s} u^{-2i}, a^{-s} u^{-2i}\}.$$  In particular, if $a^s u^{2i}\not\in Z(G)$ then $|\class(a^s u^{2i})| = 2$ if $2\mid s$, and $|\class(a^s u^{2i})| = 4$ if $2\nmid s$.  There are $2^{l-3}k-2$ such non-singleton classes with $2\mid s$, and $2^{l-3}\cdot \frac{k}{2}$ such classes with $2\nmid s$.
      \item $$\class(a^s u^{2i-1}) = \{ a^s u^{2i-1}, a^{s+2^{l-1}}u^{2i-1}, a^{-s}u^{1-2i}, a^{-s + 2^{l-1}} u^{1-2i}\}.$$  In particular, $|\class(a^s u^{2i-1})| = 4$ and there are $2^{l-3}k$ distinct conjugacy classes of this form.
      \item $$\class(a^s u^{i} v ) = \bigcup_{r,t}\left\{ a^{2r+s}u^{4t+i}v, a^{2r+n_2 s}u^{4t-i}v, a^{2r-s}u^{4t+2-i}v, a^{2r + n_1 s} u^{4t+2+i}v\right\}.$$
      There are $2^{l-2}k$ elements in each such conjugacy class, and there are $4$ distinct conjugacy classes of this form, determined by the parity of $s$ and $i$.
    \end{enumerate}
    Subsequently, $G$ has a total of $6+ 5\cdot 2^{l-3}\frac{k}{2}$ distinct conjugacy classes, including the singleton classes.
\end{prop}
\begin{proof}
  The singleton classes are just the central elements, of which there are 4 by Lemma \ref{center}.

  For any $i,j,p,q,r,s\in\BZ$ we have
  \begin{eqnarray}
    (a^p u^q v^r)(a^s u^i v^j) (v^r u^{-q} a^{-p}) &=& a^p u^q a^{n_2^r s}u^{(-1)^r i-(-1)^j q}a^{-n_2^j p}v^j\nonumber\\
    &=& a^p a^{n_1^q(n_2^r s)}u^{(-1)^r i+(1-(-1)^j)q}a^{-n_2^j p}v^j\nonumber\\
    &=& a^{(1-n_1^i n_2^j)p+n_1^q(n_2^r s)}u^{(-1)^r i+(1-(-1)^j)q}v^j\label{genconj}.
  \end{eqnarray}
  Breaking down into the cases of the parity of $p,q,r,i,j$ and using Corollary \ref{basicident} yields the desired set equalities after a bit of tiresome but elementary computation.  The cardinality statements for i) and ii) are then applications of Corollary \ref{basicident} and simple counting arguments.

  For the remaining claims in iii), we need only show that $$\class(a^s u^i v) = \class(a^p u^j v) \iff s\equiv p\bmod 2 \wedge i\equiv j\bmod 2.$$
  It is clear from the set membership that $$a^p u^j v\in\class(a^s u^i v) \Rightarrow s\equiv p \bmod 2 \wedge i\equiv j\bmod 2.$$
  So suppose that $i\equiv j\bmod 2$ and $s\equiv p\bmod 2$.

  If $i\equiv j\bmod 4$, then there are $r,t\in\BZ$ such that $a^p u^j v = a^{2r+s}u^{4t+i}v$, whence $a^p u^j v \in\class(a^s u^i v)$.  Else, there is a $t\in\BZ$ such that $4t+2+i\equiv j\bmod k$.  We claim that there is an $r\in\BZ$ with $2r+n_1 s\equiv p\bmod 2^l$.  Then we can again conclude that $a^p u^j v\in\class(a^s u^i v)$.  To this end,
  \begin{eqnarray*}
    2r+n_1s\equiv p \bmod 2^l &\iff& 2r\equiv p-n_1s \bmod 2^l.
  \end{eqnarray*}
  Since $n_1$ is odd we have $2\mid p-n_1 s$, so solutions to this equation exist.  This completes the proof.
\end{proof}

\begin{cor}\label{completelyreal}
    Let $\G$ be as in Definition \ref{presentation}.  Then $\G$ is completely real.
\end{cor}
\begin{proof}
  For elements of the form $a^s u^i v$, Corollary \ref{basicident}.vii shows that $a^s u^i v$ is self-inverse $\iff$ $2\mid s\vee 2\nmid i$.  So the non-trivial case for such elements is $2\nmid s \wedge 2\mid i$, where the element has order 4.  In this case, by the same Corollary,
  $$(a^s u^i v)^{-1} = (a^s u^i v)^3 = a^{s+2^{l-1}}u^i v.$$
  In the notation of Proposition \ref{classes}.iii, if we take $t=i/2$ and $r=-n_2 s$, we have
  $$a^{2r+n_2 s}u^{4t-i}v=a^{-n_2 s}u^ i v\in\class(a^s u^i v),$$
  and by Corollary \ref{basicident} we have
  $$a^{-n_2 s}u^ i v = a^{s+2^{l-1}}u^i v = (a^s u^i v)^{-1}.$$

  So now we need only prove that $a^s u^i$ is conjugate to its own inverse for all $s,i$.  This is obvious for the central elements, as those all have order 2 and are thus their own inverses.  Generally, we have
  $$(a^s u^i)^{-1} = u^{-i} a^{-s} = a^{-n_1^i s}u^{-i}.$$
  Using Corollary \ref{basicident} again, we conclude that this element is in $\class(a^s u^i)$, as given in Proposition \ref{classes}.i-ii.
\end{proof}
\remark If we remove the assumption that $4\mid k$ in the definition of $\G$, then the resulting group is no longer completely real in general.  The essential difference is that $\class(a^s u^{2i-1}) = \{ a^s u^{2i-1}, a^{s+2^{l-1}}u^{2i-1}\}$ when $4\nmid k$.  Such an element fails to be conjugate to its inverse precisely when $$2i\not\equiv 1\bmod \frac{k}{2} \vee \left(2^{l-1}\nmid s \wedge s(2^{l-2}-1)\not\equiv 2^{l-2}\bmod 2^{l-1}\right).$$\newline
\indent We complete this section by computing the centralizers for our groups.

\begin{prop}\label{centralizers}
    Let $G=\G$ and $n_1$ be as in Definition \ref{presentation}.  Define groups of the form $\BZ_x\rtimes_n \BZ_{2\cdot y}$ as in \cite[Def. 2.1]{K}.  The centralizers of non-central elements in $G$ are given as follows.

    \begin{enumerate}
      \item
      $$C_G(a^s u^{2i}) = \left\{ \begin{array}{lll}
        \cyc{a,u} \cong \BZ_{2^l}\rtimes_{n_1}\BZ_{2\cdot k/2} &;& 2\mid s\\
        \cyc{a,u^2} \cong \BZ_{2^l}\times \BZ_{k/2} &;& 2\nmid s
      \end{array}\right.$$
      In particular, $C_G(a^s u^{2i})$ is a normal subgroup of $G$.

      \item
      $$C_G(a^s u^{2i-1}) = \left\{ \begin{array}{lll}
        \cyc{a^2,u}\cong \BZ_{2^{l-1}}\times \BZ_k &;& 2\mid s\\
        \cyc{a^2, au}\footnotemark &;& 2\nmid s
      \end{array}\right.$$
      In particular, $C_G(a^s u^{2i-1})$ is an abelian normal subgroup of $G$ of index 4.

      \item
      $$C_G(a^s u^i v) = \cyc{a^{2^{l-1}},u^{k/2},a^s u^i v}.$$
      In particular, $C_G(a^s u^i v)$ is an abelian subgroup of $G$ of order 8.  Indeed,
      $$C_G(a^s u^i v) \cong \left\{ \begin{array}{lll}
        \BZ_4\times\BZ_2 &;& 2\mid i \wedge 2\nmid s\\
        \BZ_2^3 &;& 2\nmid i \vee 2\mid s
      \end{array}\right.$$
    \end{enumerate}
\end{prop}
\footnotetext{The exact isomorphism type of this subgroup depends on the group $G$, but will not be needed here.}
\begin{proof}
  We have previously computed what a generic conjugate in $G$ looks like in equation (\ref{genconj}).  So we break things down by cases to obtain our centralizers.

  When $2\mid j$ and $2\mid i$, we have
  $$(\ref{genconj}) = a^{n_1^q n_2^r s}u^{(-1)^r i}.$$
  This is equal to $a^s u^i$ $\iff$ $(2\mid r \wedge n_1^q s\equiv s\bmod 2^l) \vee (2\nmid r \wedge k\mid 2i \wedge n_1^q n_2 s\equiv s\bmod 2^l)$.  For the $2\nmid r$ case, the equivalence holds $\iff$ $2^{l-1}\mid s$.  Since also $k\mid 2i$ in this case, this forces $a^s u^i\in Z(G)$.  Then for the $2\mid r$ case, the equivalence always holds if $2\mid q$.  When $q$ is odd, by Corollary \ref{basicident} the equivalence holds precisely when $2\mid s$.  This gives the identities of part i).  The isomorphisms are clear.  When $2\mid s$ the normality is clear by definition of the action of $v$ on $u$ and $a$.  For $2\nmid s$, we easily find that
  \begin{eqnarray*}
     a u^2 a^{-1} &=& u^2\\
     v u^2 v &=& u^{-2}.
  \end{eqnarray*}
  Thus the centralizer is again normal as claimed.

  Now suppose $2\mid j$ and $2\nmid i$.  Then
  $$(\ref{genconj}) = a^{(1-n_1)p+n_1^q n_2^r s} u^{(-1)^r i}.$$
  This is equal to $a^s u^i$ $\iff$ $( 2\mid r \wedge (n_1-1)p\equiv (n_1^q-1)s\bmod 2^l)$.  For $2\nmid q$, this holds precisely when $p\equiv s\bmod 2$.  For $2\mid q$, it holds precisely when $2\mid p$.  This gives the equalities in part ii).  The index follows immediately from Proposition \ref{classes}.  By Corollary \ref{basicident}, the centralizer is abelian when $2\mid s$.  We also have that
  $$a^{2} au a^{-2} = a^{3-2n_1}u = a^{3-2(2^{l-1}+1)}u = au,$$
  so in this case the centralizer is again abelian.  For normality, we compute
  \begin{eqnarray*}
     ua^2 u^{-1} &=& a^{2 n_1}\\
     v a^2 v &=& a^{2 n_2}\\
     a(au)a^{-1} &=& a^{2-n_1} u = a^{2(1-2^{l-2})}\cdot au\\
     u (au) u^{-1} &=& ua = a^{n_1}u = a^{2^{l-1}}\cdot au\\
     v(au) v &=& a^{n_2}u^{-1}\\
     &=& u^{-1} a^{n_1 n_2} = u^{-1}a^{-1} = (au)^{-1}.
  \end{eqnarray*}
  Thus the centralizers are normal as claimed.

  Lastly, suppose that $2\nmid j$.  Then
  $$(\ref{genconj}) = a^{(1-n_1^i n_2)p+n_1^q n_2^r s}u^{(-1)^r i +2q} v.$$
  This is equal to $a^s u^i v$ if and only if
  \begin{eqnarray}\label{vcenteqn}
    (-1)^r i + 2q \equiv i\bmod k \wedge (n_1^i n_2-1)p \equiv (n_1^q n_2^r-1)s\bmod 2^l.
  \end{eqnarray}
  We now break down these congruences into cases.

  When $2\mid r$, we have
  \begin{eqnarray*}
    (\ref{vcenteqn}) &\iff& 2q\equiv 0\bmod k \wedge (n_1^i n_2-1)p \equiv (n_1^q-1)s\bmod 2^l\\
    &\iff& q\equiv 0\bmod \frac{k}{2} \wedge (n_1^i n_2-1)p\equiv 0\bmod 2^l\\
    &\iff& q\equiv 0\bmod \frac{k}{2} \wedge 2^{l-1}\mid p.
  \end{eqnarray*}
  In the second line we have used that $4\mid k \wedge q\equiv 0\bmod k/2$ forces $2\mid q$.  These equivalences simply equate to the obvious statement that $\cyc{a^{2^{l-1}},u^{k/2}}=Z(G)\subseteq C_G(a^s u^i v)$.

  So for $2\nmid r$, we he have
  \begin{eqnarray*}
    (\ref{vcenteqn}) &\iff& 2q\equiv 2i\bmod k \wedge (n_1^i n_2-1)p\equiv (n_1^q n_2-1)s\bmod 2^l\\
    &\iff& q\equiv i\bmod \frac{k}{2} \wedge (n_1^i n_2-1)p\equiv (n_1^i n_2-1)s\bmod 2^l\\
    &\iff& q\equiv i\bmod \frac{k}{2} \wedge p\equiv s \bmod 2^{l-1}
  \end{eqnarray*}
  In the second line we have used that $4\mid k\wedge q\equiv i\bmod \frac{k}{2}$ forces $q\equiv i\bmod 2$. These last equivalences amount to the statement that $\cyc{a^{2^{l-1}},u^{k/2},a^s u^i v}\subseteq C_G(a^s u^i v)$.  Since this exhausts all possible cases, we in fact conclude that $\cyc{a^{2^{l-1}},u^{k/2},a^s u^i v} = C_G(a^s u^i v)$.  Since two of the generators are in the center of $G$, it is clear that these groups are abelian.  The isomorphisms follow from Corollary \ref{basicident} and Lemma \ref{center}, and the index follows from Proposition \ref{classes}.iii.
\end{proof}


\section{Representations of $\G$}\label{repsect}
We proceed now to determine the character theory of the groups given by Definition \ref{presentation}.

Proposition 2.7.i in part says that $H_1=\cyc{a,u^2}$ is an abelian normal subgroup of $\G$.  If $\alpha$ is a generating irreducible character for $\cyc{a}$ and $\beta$ is a generating irreducible character for $\cyc{u^2}$, then the irreducible characters of $H_1$ are of the form $\alpha^r\otimes \beta^t$ for some $r,t\in\BZ$.  We now show that inducing these characters up to $\G$ yield 4-dimensional representations of $\G$ which are often irreducible.
\begin{thm}\label{4dimreps}
  Let $G=\G$ and $n_1$ be defined as in Definition \ref{presentation}.  Let $\varphi_{r,t}=\alpha^r\otimes\beta^t \in \widehat{H_1}$ be any irreducible character of $H_1$, as above.  Then the induced character $\varphi_{r,t}^G$ has dimension 4.  Furthermore, $\varphi_{r,t}^G$ is irreducible if and only if $2\nmid r$.  Amongst these irreducible characters, there are $2^{l-3}\frac{k}{2}$ isomorphism classes.
\end{thm}
\begin{proof}
  There are clearly $2^{l-1}k$ distinct characters of $H_1$.  Let $\psi=\alpha^r\otimes\beta^t$ and $\phi=\alpha^q\otimes \beta^s$ be any two characters of $H_1$.  Using Proposition \ref{classes}, it is easily verified that $\psi^G$ has dimension 4.  By normality, $\psi^G(x)=0$ whenever $x\not\in H_1$.

  Define $e\colon \{1,2,3,4\}\to \{\pm 1, n_1, n_2\}$ by
  \begin{eqnarray}\label{edef}
    e(x) &=& \left\{ \begin{array}{cll} -1 &;& x=1\\
        1 &;& x=2\\
        n_2 &;& x=3\\
        n_3 &;& x=4 \end{array}\right.
  \end{eqnarray}
  Then, by Proposition \ref{classes} we have
  \begin{eqnarray*}
    \ip{\psi^G}{\psi^G} &=& \frac{1}{|G|} \sum_{i=0}^{k/2}\sum_{j=0}^{2^l -1} \sum_{x,y=1}^r \psi(a^{e(x)j} u^{(-1)^x 2i}) \psi(a^{-e(y)j} u^{(-1)^{y+1} 2i})\\
    &=& \frac{1}{|G|}\sum_{x,y=1}^4 \left( \sum_{i=0}^{k/2}\sum_{j=0}^{2^l -1} \alpha^{re(x)}(a^j) \beta^{(-1)^x s}(u^{2i})\alpha^{re(y)}(a^{-j}) \beta^{(-1)^y s}(u^{-2i})\right)\\
    &=& \frac{1}{|G|} \sum_{x,y=1}^4 \left( \sum_{i=0}^{k/2} \beta^{(-1)^x s}(u^{2i}) \beta^{(-1)^y s}(u^{-2i}) \left( \sum_{j=0}^{2^l -1} \alpha^{re(x)}(a^j) \alpha^{re(y)}(a^{-j})\right) \right)\\
    &=& \frac{1}{4} \sum_{x,y=1}^r \ip{\beta^{(-1)^x s}}{\beta^{(-1)^y s}} \ip{\alpha^{r e(x)}}{\alpha^{r e(y)}}
  \end{eqnarray*}
  This latter expression is equal to $1$ $\iff$ $2\nmid r$.  Thus $\psi^G$ is irreducible if and only if $2\nmid r$, and so there are $2^{l-2}k$ irreducible characters amongst such induced characters.

  To determine how many of these characters are distinct, we suppose that $2\nmid r$ and $2\nmid q$ and similarly compute
  \begin{eqnarray*}
    \ip{\psi^G}{\phi^G} &=& \frac{1}{4} \sum_{x,y=1}^4 \ip{\beta^{(-1)^x s}}{ \beta^{(-1)^y t}} \ip{\alpha^{e(x)q}}{\alpha^{e(y)r}}.
  \end{eqnarray*}
  This expression is non-zero, or equivalently $\psi^G \cong \phi^G$, if and only if there are \newline $1\leq x,y\leq 4$ solving the equations
  $$s\equiv (-1)^{x+y}t \bmod \frac{k}{2} \ \mbox{and} \ q\equiv e(x)e(y) r \bmod 2^l.$$
  This in turn is equivalent to one of following four pairs of congruences holding
  \begin{eqnarray*}
    s \equiv\ t \bmod \frac{k}{2} &\mbox{and}& \ q\equiv r \bmod 2^l\\
    s \equiv -t \bmod \frac{k}{2} &\mbox{and}& \ q\equiv -r \bmod 2^l\\
    s \equiv\ t \bmod \frac{k}{2} &\mbox{and}& \ q\equiv n_1 r \bmod 2^l\\
    s \equiv -t \bmod \frac{k}{2} &\mbox{and}& \ q\equiv n_2 r \bmod 2^l.
  \end{eqnarray*}
  Since these congruences respect the parity conditions $2\nmid r$ and $2\nmid q$, we conclude that there are four copies (up to isomorphism) of each irreducible character amongst these induced characters.  In particular, we obtain $2^{l-3}\frac{k}{2}$ distinct 4-dimensional irreducible characters in this fashion.
\end{proof}

Define $H_2 = \cyc{a,u}\cong \BZ_{2^{l}}\rtimes_{n_1} \BZ_{2\cdot k/2}$, a normal subgroup of $G$ by Proposition \ref{centralizers}.i.  By \cite{K}, the linear characters of $H_2$ are parameterized by $\phi_{r,t}= \gamma^r\otimes \rho^t\in\widehat{\BZ}_{2^{l-1}}\otimes \widehat{\BZ}_k$, with $\gamma$ and $\rho$ generating $\widehat{\BZ}_{2^{l-1}}$ and $\widehat{\BZ}_k$ respectively.  Evaluation on $H_2$ is then given by using any surjective homomorphism $\cyc{a,u}\twoheadrightarrow \cyc{a^2, u} \cong \BZ_{2^{l-1}}\times \BZ_k$.  Equivalently, these characters arise from the linear characters of the quotient $G/\cyc{a^{2^{l-1}},v}$.

\begin{thm}\label{2dimreps}
    Let $G=\G$ and $n_1$ be defined as in Definition \ref{presentation}.  Let $\phi_{r,t}=\gamma^r\otimes \rho^t\in \widehat{H_2}$ be any linear character of $H_2$, as above.  For any such linear character, the induced character $\phi_{r,t}^G$ has dimension 2.  Of these induced characters, $2^{l-1}k-4$ are irreducible, of which there are $2^{l-2}k-2$ distinct equivalence classes.
\end{thm}
\begin{proof}
    There are clearly $2^{l-1}k$ distinct linear characters of $H_2$.  Let $\phi=\alpha^r\otimes\beta^t$ and $\theta=\alpha^q\otimes \beta^s$ be any two linear characters of $H_2$.  Using Propositions \ref{classes} and \ref{centralizers}, it is easily verified that $\phi^G$ has dimension 2.  Indeed, we have
    \begin{eqnarray*}
      \phi^G(x) &=& \left\{ \begin{array}{cll}
        \phi(x) + \phi(x^{-1}) &;& x\in H\\
        0 &;& x\not\in H \end{array}\right..
    \end{eqnarray*}

    Subsequently, we find that
    \begin{eqnarray*}
      \ip{\phi^G}{\theta^G} &=& \frac{1}{|G|}\sum_{x\in H} \left(\phi(x)+\phi(x^{-1})\right) \left(\theta(x) + \theta(x^{-1})\right)\\
      &=& \frac{1}{2}\ip{\phi}{\theta} + \frac{1}{2}\ip{\theta}{\phi} + \frac{1}{2}\ip{\phi\theta}{1} + \frac{1}{2}\ip{\theta\phi}{1}\\
      &=& \ip{\phi}{\theta} + \ip{\phi\theta}{1}.
    \end{eqnarray*}
    Specializing to $\theta=\phi$, we see that this expression is 1, or equivalently that $\phi^G$ is irreducible, if and only if $\phi^2\neq 1$.  There are exactly four $\phi$ with $\phi^2=1$, given by $r\equiv 0 \bmod 2^{l-2}$ and $t\equiv 0\bmod \frac{k}{2}$.  This means there are $2^{l-1}k-4$ irreducible characters amongst these induced characters.  We then see that, for a generic $\theta$, $\ip{\phi^G}{\theta^G} = 0 \iff \phi\neq\theta \wedge \phi\theta\neq 1$.  So when $\phi$, $\theta$ are distinct with $\phi^G$ and $\theta^G$ irreducible, we conclude that they are isomorphic if and only if $\phi = \theta^{-1}$.  Thus there are $2^{l-2}k-2$ isomorphism classes amongst these irreducible characters of $G$.
\end{proof}

It turns out that these 4-dimensional and 2-dimensional characters give all of the non-linear characters of $\G$, which we now prove.

\begin{thm}\label{1dimreps}
    Let $G=\G$ be as in Definition \ref{presentation}.  Then $G$ has exactly eight (non-isomorphic) irreducible 1-dimensional characters.  They are given by the mappings $\{a,u,v\}\mapsto \{-1,1\}$.  Equivalently, they come from the linear characters of the quotient group $G/\cyc{a^2,u^2}$.  All other irreducible characters of $G$ are given by those in Theorems \ref{4dimreps} and \ref{2dimreps}.  In particular, $G$ has exactly $6+ 5\cdot 2^{l-3}\frac{k}{2}$ non-isomorphic irreducible representations.
\end{thm}
\begin{proof}
  We know that the number of distinct non-isomorphic irreducible characters of $G$ is equal to the number of conjugacy classes in $G$.  By Proposition \ref{classes}, there are thus $6 + 5\cdot 2^{l-3}\frac{k}{2}$ non-isomorphic irreducible characters of $G$.  The linear characters of the quotient $G/\cyc{a^2, u^2} \cong \BZ_2^3$ clearly yield 8 non-isomorphic linear characters of $G$.  Adding these 8 to the $2^{l-3}\frac{k}{2}$ characters from Theorem \ref{4dimreps} and the $2^{l-2}k-2$ characters from Theorem \ref{2dimreps}, we have accounted for $6+ 5\cdot 2^{l-3}\frac{k}{2}$ total distinct characters.  Thus these give all of the non-isomorphic irreducible characters of $G$.
\end{proof}

Since we have now accounted for every irreducible character of $G$ by inducing from certain linear characters, we recall the following definition.

\begin{df}\label{Mgroupdef}\cite[Def. 5.10]{I}
    Let $G$ be a finite group.  We say that $G$ is an $M$-group if for every irreducible character $\varphi$ of $G$ there is a subgroup $H\subseteq G$ (possibly $H=1$ or $H=G$) and a linear character $\phi$ of $H$ with $\phi^G=\varphi$.
\end{df}

\begin{cor}\label{mgroups}
  Let $G=\G$ be as in Definition \ref{presentation}.  Then for every $g\in G$, $C_G(g)$ is an $M$-group.  In particular, $G$ is an $M$-group.
\end{cor}
\begin{proof}
  By Theorems \ref{4dimreps}, \ref{2dimreps}, and \ref{1dimreps}, all irreducible characters of $G$ can be obtained by inducing a linear character of some subgroup of $G$.  By Proposition \ref{centralizers}, the centralizers of non-central elements are either abelian or are isomorphic to one of the groups considered in \cite{K}.  Abelian groups are clearly $M$-groups, and it was shown in \cite{K} that the groups considered there are also $M$-groups.
\end{proof}
In fact, a somewhat stronger condition holds, as every irreducible character for $C_G(g)$ can be induced from a linear character of some \textit{normal} subgroup.


\section{Indicators of $\G$}\label{groupindsect}
Now that we know the complete representation theory of our groups, we can proceed to compute their higher indicators.  The easiest case is, of course, the linear characters.

\begin{thm}\label{1dimindic}
  Let $G=\G$ be as in Definition \ref{presentation}.  Suppose $\phi$ is a 1-dimensional character of $G$.  Then
  $$\nu_m(\phi) = \left\{\begin{array}{cll}
    1 &;& 2\mid m \vee \phi = 1\\
    0 &;& 2\nmid m \wedge \phi \neq 1
  \end{array}\right.$$
  In particular, $\nu_2(\phi)=1$ for every such $\phi$.
\end{thm}
\begin{proof}
  For any such $\phi$, we have
  $$\frac{1}{|G|}\sum_{g\in G}\phi(g^m) = \frac{1}{|G|}\sum_{g\in G}\phi^m(g) = \ip{\phi^m}{1}.$$
  By Theorem \ref{1dimreps}, $\phi$ is either the identity or has order 2.  The result follows.
\end{proof}

\begin{thm}\label{4dimindic}
    Let $\varphi_{r,t}^G$ be an irreducible 4-dimensional $G$-module, as given in Theorem \ref{4dimreps}.  Then
    \begin{eqnarray*}
    \nu_m(\varphi_{r,t}^G) &=& \ip{\varphi_{mr,mt}}{1} + \delta_{2\mid m} (-1)^{m/2}\left(2\ip{\varphi_{mr,\frac{m}{2}t}}{1} - \ip{\varphi_{mr,mt}}{1}\right) + \delta_{2\mid m} + \delta_{4\mid m}\\
    &=& \left\{ \begin{array}{cll}
        4 &;& 2^{l-1}\mid m \wedge k\mid mt\\
        2 &;& 4\mid m \wedge (2^{l-1}\nmid m \vee k\nmid t)\\
        1 &;& 2\mid m \wedge 4\nmid m\\
        0 &;& 2\nmid m \end{array}\right..
    \end{eqnarray*}
    In particular, $\nu_2(\varphi_{r,t}^G)=1$ and $\nu_m(\varphi_{r,t}^G)\geq 0$ for every $m$.
\end{thm}
\begin{proof}
  Suppose for now that the stated formula is valid.  Since $2\nmid r$ by Theorem \ref{4dimreps}, we observe that that $2\nmid m \Rightarrow \nu_m(\varphi_{r,t}^G) = 0$.  If $4\mid m$, then
  \begin{eqnarray*}
    \nu_m(\varphi_{r,t}^G) = 2 + 2\ip{\varphi_{mr,\frac{m}{2}t}}{1} \in\{2,4\}.
  \end{eqnarray*}
  Indeed, in this case $\nu_m(\varphi_{r,t}^G)=4$ $\iff$ $2^{l-1}\mid m \wedge k\mid mt$.   Additionally, if $2\mid m$ but $4\nmid m$ then
  \begin{eqnarray*}
    \nu_m(\varphi_{r,t}^G) = 2\ip{\varphi_{mr,mt}}{1} - 2 \ip{\varphi_{mr,\frac{m}{2}t}}{1} + 1 = 1\\
  \end{eqnarray*}
  where we have again used the fact that $2\nmid r$ to conclude that the inner products are all 0.  In particular, $\nu_2(\varphi_{r,t}^G)=1$.  This establishes the last two claims in the theorem.

  To establish the formula, first recall that $\varphi_{r,t}^G(x)=0$ whenever $x\not\in H=\cyc{a,u^2}$.  Therefore, we need to compute when $m$-th powers of elements of $G$ are in $H$.  This is trivially true for all $m$ for elements of the form $a^s u^{2i}$.  By the relations of the group, $(a^s u^{2i-1})^m \in H \iff 2\mid m$, and by Corollary \ref{basicident} $(a^s u^i v)^m\in H \iff 2\mid m$.

  We thus have
  \begin{eqnarray*}
    \nu_m(\varphi_{r,t}^G) &=& \frac{1}{|G|} \sum_{s=0}^{2^l-1}\left( \sum_{i=0}^{k/2-1} \left(\varphi_{r,t}^G((a^s u^{2i})^m) + \varphi_{r,t}^G((a^s u^{2i-1})^m)\right) + \sum_{i=0}^{k-1} \varphi_{r,t}^G((a^s u^i v)^m)\right)\\
    &=& \frac{1}{|G|}\left( \delta_{2\mid m} 2^{l-2}k \varphi_{r,t}^G(a^{2^{l-2}m}) + \delta_{2\mid m} 3\cdot 2^{l-2}k\varphi_{r,t}^G(1)\ + \ \right.\\
    && \indent +\sum_{s=0}^{2^{l-1}}\sum_{i=0}^{k/2-1}\left( \varphi_{r,t}^G(a^{ms}u^{2mi}) + \delta_{2\mid m} \varphi_{r,t}^G(a^{2^{l-2}m} a^{ms} u^{m(2i-1)})\right)\Big)\\
    &=& \delta_{2\mid m}\left(\frac{3}{2}+(-1)^{m/2}\frac{1}{2}\right) + 2\ip{\varphi_{mr,mt}}{1} + \sum_{s=1}^{2^l} \sum_{i=1}^{k/2} \varphi_{r,t}^G(a^{ms}u^{2mi-m})\\
    &=& \delta_{2\mid m} + \delta_{4\mid m} + 2\ip{\varphi_{mr,mt}}{1}\ + \\
    && \indent +\delta_{2\mid m}(-1)^{m/2}\frac{1}{|G|}\sum_{s=1}^{2^l} \left(  \sum_{i=1}^k\varphi_{r,t}^G(a^{ms} u^{\frac{m}{2}2i}) - \sum_{i=1}^{k/2}\varphi_{r,t}^G(a^{ms} u^{\frac{m}{2}4i})\right)\\
    &=& \ip{\varphi_{mr,mt}}{1} + \delta_{2\mid m} (-1)^{m/2}\left(2\ip{\varphi_{mr,\frac{m}{2}t}}{1} - \ip{\varphi_{mr,mt}}{1}\right) + \delta_{2\mid m} + \delta_{4\mid m},
  \end{eqnarray*}
  which is the desired formula.
\end{proof}

\begin{thm}\label{2dimindic}
    Let $\psi_{r,t}^G$ be an irreducible 2-dimensional representation of $G$, as given in Theorem \ref{2dimreps}.  Then
    $$\nu_m(\psi_{r,t}^G) = \ip{\psi_{mr,mt}}{1} + \delta_{2\mid m} = \left\{ \begin{array}{cll}
        2 &;& 2\mid m \wedge \psi_{mr,mt}=1\\
        1 &;& 2\mid m \wedge \psi_{mr,mt}\neq 1\\
        0 &;& 2\nmid m
    \end{array}\right..$$
    In particular, $\nu_2(\psi_{r,t}^G)=1$ and $\nu_m(\psi_{r,t}^G)\geq 0$ for every $m$.
\end{thm}
\begin{proof}
  The last two claims are immediate from the formula and the fact that $\psi_{r,t}$ does not have order two by the proof of Theorem \ref{2dimreps}.

  To obtain the formula, we proceed as we did in the proof of Theorem \ref{4dimindic}.  Things are a bit simpler here since $\psi_{r,t}(a^{2^{l-1}})=1$.  We have
  \begin{eqnarray*}
    \nu_m(\psi_{r,t}^G) &=& \frac{1}{|G|}\sum_{s=1}^{2^l}\sum_{i=1}^k \left( \psi_{r,t}^G((a^s u^i)^m) + \psi_{r,t}^G((a^s u^i v)^m)\right)\\
    &=& \frac{1}{|G|}\Big( \delta_{2\mid m}\psi_{r,t}(a^{2^{l-1}\frac{m}{2}})2^{l-1}k + \delta_{2\mid m}3\cdot2^{l-1}k\ + \\
    &&\indent +\sum_{s=1}^{2^l} \sum_{i=1}^{k/2}\left( \psi_{r,t}^G(a^{ms}u^{2mi}) + \psi_{r,t}^G(a^{2^{l-1}\frac{m}{2}}a^{ms} u^{m(2i-1)})\right)\Big)\\
    &=& \delta_{2\mid m} + \frac{1}{|G|} \sum_{s=1}^{2^l}\sum_{i=1}^{k} \psi_{r,t}^G(a^{ms} u^{mi})\\
    &=& \delta_{2\mid m} + \ip{\psi_{mr,mt}}{1},
  \end{eqnarray*}
  as desired.
\end{proof}

Combing our results so far, we get the following.

\begin{thm}\label{groupsummary}
    Let $G\cong\G$, with $\G$ as in Definition \ref{presentation}.  Then the following hold:
        \begin{enumerate}
            \item $G$ is completely real.
            \item $G$ is generated by involutions.
            \item $\forall g\in G$ $C_G(g)$ is an $M$-group.  In particular, $G$ is an $M$-group.
            \item $G$ is totally orthogonal.
            \item $\nu_m(V)\in\BN\cup\{0\}$ for every $m$ and every $G$-module $V$.
        \end{enumerate}
\end{thm}
\begin{proof}
  Part i) is Corollary \ref{completelyreal}.  Part ii) is Corollary \ref{invgen}.  Part iii) is Corollary \ref{mgroups}.  Parts iv) and v) follow immediately from Theorems \ref{1dimindic}, \ref{2dimindic}, and \ref{4dimindic}
\end{proof}

\section{The sets $G_m(x)$}\label{gmsect}
We continue to let $G=\G$ be as in Definition \ref{presentation}.  We wish now to proceed to compute the indicators for the irreducible modules over $\D(G)$.  We use the notation of Proposition \ref{MOrho} to denote the irreducible modules over $\D(G)$.  By Proposition \ref{FxF}, the indicators for a module $M(g,\rho)$ with $g\in Z(G)$ are entirely determined by the order of $g$ and the indicators of the $G$-module given by $\rho$, which we have already computed in the previous section.  Thus we will subsequently focus on the indicators of modules corresponding to non-singleton conjugacy classes.

By Corollary \ref{cor34}, we will need to compute the sets $G_m(x)$ from Definition \ref{Gmsetdef}.  Determining these sets is the goal of this section.  The remaining sections are dedicated to computing the indicator for a particular "type" of module, which we describe now.

\begin{df}\label{typedef}
  Let $G=\G$ be as in Definition \ref{presentation}.  Fix $g\in G$ with $g\not\in Z(G)$. Let $V=M(\class(g),\eta)$ be any irreducible $\D(G)$-module, as in Proposition \ref{MOrho}.
  \begin{enumerate}
    \item We say that $V$ is Type I if $g=a^s u^{2i}$.
    \item We say that $V$ is Type II if $g=a^s u^{2i-1}$.
    \item We say that $V$ is Type III if $g=a^s u^i v$
  \end{enumerate}
  For Type I modules, we further say that $V$ is even or odd according to whether $s$ is even or odd.
\end{df}

Before we compute the membership of the sets $G_m(x)$, we start off with a preliminary lemma that will assist our computation.
\begin{lem}\label{gmsetlem}
    Let $G=\G$ be as in Definition \ref{presentation}. Then
     $$(a^x u^y v)^j = a^{\left( \ceil{\frac{j}{2}}+\floor{\frac{j}{2}}n_1^y n_2\right)x} u^{\delta_{2\nmid j}\ y} v^j.$$
\end{lem}
\begin{proof}
  This is easily verified to be equivalent to Corollary \ref{basicident}.vii.
\end{proof}
\begin{prop}\label{gmsets}
  Let $G=\G$ be as in Definition \ref{presentation}.  Let $m\in\BN$ and define the sets $G_m(x)$ as in Definition \ref{Gmsetdef}.  Then for $i,j,r,s\in\BZ$ we have the following.
  \begin{enumerate}
    \item $a^s u^{j} \in G_m(a^r u^{i}) \iff k\mid mi \wedge 2^l\mid mr$
    \item $a^s u^j v \in G_m(a^r u^{2i}) \iff (a^r u^{2i}=1) \vee \left( 2\mid m \wedge (2\nmid j \vee 4\mid mr)\right)$
    \item $a^s u^j v \in G_m(a^r u^{2i-1}) \iff 2\mid m \wedge \left( \left(2\mid j\wedge 4\mid ms\right) \vee \left( 2\nmid j \wedge 4\mid m(r+s)\right)\right)$
    \item If $2\nmid m$ or $k\nmid 2mj$, then $a^s u^{2j}\not\in G_m(a^r u^{2i}v)$.  Else, if $2\mid m$ and $k\mid 2mj$, then
        $$a^s u^{2j} \in G_m(a^r u^{2i} v) \iff \left( 4\mid mr \wedge 2^l\mid ms\right) \vee \left(4\nmid mr \wedge s\equiv 2^{l-2}\bmod 2^{l-1}\right).$$
    \item $a^s u^{2j-1}\in G_m(a^r u^{2i} v) \iff 4\mid m \wedge k\mid m(2j-1)\wedge 2^l \mid ms$
    \item $a^s u^{2j} \in G_m(a^r u^{2i-1} v) \iff 2\mid m\wedge k\mid 2mj\wedge 2^l\mid ms$
    \item $a^s u^{2j-1}\in G_m(a^r u^{2i-1} v) \iff 4\mid m \wedge k\mid m(2j-1) \wedge 2^l \mid ms$
    \item If $2\mid j$ and $2\mid i$, then $$a^s u^j v\in G_m(a^r u^i v) \iff 2\mid m \wedge k\mid m(i-j) \wedge mr\equiv (1-2^{l-2})ms \bmod 2^l.$$
    \item If $2\nmid j$ and $2\mid i$, then $$a^s u^j v\in G_m(a^r u^i v) \iff 4\mid m \wedge k\mid m(i-j) \wedge 2^l\mid m(r-s)$$
    \item If $2\mid j$ and $2\nmid i$, then $$a^s u^j v \in G_m(a^r u^i v) \iff 4\mid m \wedge k\mid m(i-j) \wedge 2^l\mid m(r-s).$$
    \item If $2\nmid j$ and $2\nmid i$, then $$a^s u^j v\in G_m(a^r u^i v) \iff 2\mid m \wedge k\mid m(i-j) \wedge 2^l\mid m(r-s).$$
  \end{enumerate}
\end{prop}
\remark The condition "$4\mid m$" in parts (v),(vii),(ix), and (x) is technically superfluous.  In each case it is a consequence of the condition requiring that $k$ divide a particular expression involving $m$, and our assumption that $4\mid k$.  We explicitly state this condition here since it is useful to keep in mind when computing the values of indicators.
\begin{proof}
  \begin{enumerate}
  \item  Since $\cyc{a,u}$ is one of the groups considered in \cite{K}, this is a straightforward consequence of \cite[Prop. 4.1]{K} and Corollary \ref{basicident}.

  \item By equation (\ref{genconj}) and Lemma \ref{gmsetlem},
  $$(a^x u^y v)^{j} (a^r u^{2i}) (a^x u^y v)^j = a^{(n_2 n_1^y)^j r} u^{(-1)^j 2i}.$$
  Therefore
  \begin{eqnarray}
    \prod_{j=0}^{m-1} (a^x u^y v)^{j} (a^r u^{2i}) (a^x u^y v)^j &=& \prod_{j=0}^{m-1}a^{(n_2 n_1^y)^j r} u^{(-1)^j 2i}\nonumber\\
    &=& a^{\sum_{j=0}^{m-1} (n_2 n_1^y)^j r} u^{\sum_{j=0}^{m-1}(-1)^j2i}\nonumber\\
    &=& a^{\sum_{j=0}^{m-1} (n_2 n_1^y)^j r} u^{\delta_{2\nmid m}2i}.\label{gmeqn1}
  \end{eqnarray}
  So for this element to equal 1, we see from the power on $u$ that we must have either $2\mid m$ or $k\mid 2i$.  We consider each case to determine when the power on $a$ gives the identity element.

  If $2\nmid m$, then for $(\ref{gmeqn1})=1$ we must have
  $$\left(1+\frac{m-1}{2}(1+n_2 n_1^y)\right)r\equiv 0\bmod 2^l \iff 2^l \mid r,$$
  where the equivalence follows from the fact that $(1+n_2 n_1^y)$ is even for every $y$.  Since we also have $k\mid 2i$ if (\ref{gmeqn1})=1, it immediately follows that $a^r u^{2i} = 1$.

  So suppose now that $2\mid m$.  Then the condition for the power on $a$ in (\ref{gmeqn1}) is
  $$\frac{m}{2}(1+n_2 n_1^y)r\equiv 0\bmod 2^l.$$
  If $2\mid y$, this is equivalent to $2^{l-2}mr \equiv 0\bmod 2^l$ $\iff$ $4\mid mr$.  On the other hand, if $2\nmid y$, the equation always holds.  Combined, this gives part ii).

  \item By equation (\ref{genconj}) and Lemma \ref{gmsetlem},
  $$(a^x u^y v)^{-j} (a^s u^{2i-1}) (a^x u^y v)^j = a^{2^{l-1}\left( \ceil{\frac{j}{2}} + \floor{\frac{j}{2}} n_1^y n_2\right)x} a^{(n_2 n_1^y)^j r} u^{(-1)^j(2i-1)}.$$
  Therefore
  \begin{eqnarray*}
    \prod_{j=0}^{m-1}(a^x u^y v)^{-j} (a^s u^{2i-1}) (a^x u^y v) &=& \prod_{j=0}^{m-1} a^{2^{l-1}\left( \ceil{\frac{j}{2}} + \floor{\frac{j}{2}} n_1^y n_2\right)x} a^{(n_2 n_1^y)^j r} u^{(-1)^j(2i-1)}\\
    &=& a^{\sum_{j=0}^{m-1} 2^{l-1}jx} a^{\sum_{j=0}^{m-1} n_2^j n_1^{(y+1)j}r} u^{\sum_{j=0}^{m-1} (-1)^j(2i-1)}\\
    &=& a^{2^{l-1}\floor{\frac{m}{2}}x} a^{\sum_{j=0}^{m-1} n_2^j n_1^{(y+1)j}r} u^{\delta_{2\mid m}(2i-1)}.
  \end{eqnarray*}
  So for this element to equal 1, we must have $2\mid m$.  Supposing that $2\mid m$, we now compute the power on $a$ and when it yields the element 1.  If we had $2\mid y$, then the power becomes
  $$2^{l-1}\frac{m}{2}x + \sum_{j=0}^{m-1}(-1)^j r \equiv 2^{l-1}\frac{mx}{2}\bmod 2^l,$$
  and this is divisible by $2^l$ if and only if $4\mid mx$.  On the other hand, if $2\nmid y$, then the power becomes
  $$2^{l-1}\frac{m}{2}x + \sum_{j=0}^{m-1}n_2^j r \equiv 2^{l-1}\frac{m}{2}(r+x),$$
  and this is divisible by $2^l$ if and only if $4\mid m(x+r)$.

  \item[iv-vii)] By equation (\ref{genconj}), we have
  $$(a^x u^y)^{-j}(a^r u^i v)(a^x u^y)^j = a^{n_1^{jy}\left(r+(n_2 n_1^i -1)(\ceil{j/2}+\floor{j/2}n_1^y)x\right)} u^{i-2jy} v.$$
  For notational convenience, set $N_j= r+(n_2 n_1^i -1)(\ceil{j/2}+\floor{j/2}n_1^y)x$.  Then
  \begin{eqnarray}\label{gmeqn2}
    \prod_{j=0}^{m-1} (a^x u^y)^{-j}(a^r u^i v)(a^x u^y)^j &=& a^{\sum_{j=0}^{m-1} (n_2 n_1^i)^j n_1^{jy}N_j} u^{\sum_{j=0}^{m-1} (-1)^j (i-2jy)} v^m.
  \end{eqnarray}
  From the power on $v$, we conclude that for this element to be the identity we must have $2\mid m$.  So we suppose that $2\mid m$.  Then the power on $u$ is given by
  $$\sum_{j=0}^{m-1} (-1)^j (i-2jy) = my.$$
  Thus in order for (\ref{gmeqn2})=1 to hold we must also have $k\mid my$.  We note now that if $2\nmid y$, then this will force $4\mid m$ by assumptions on $k$.  To analyze the power on $a$, we must now break things down into cases.  We may assume, without loss of generality, that $2\mid m$ and $k\mid my$ for this analysis.

  Suppose first that $2\mid y$.  Then the power on $a$ is given by
  $$\frac{m}{2}(n_2 n_1^i+1)r + \sum_{j=0}^{m-1} (n_2 n_1^i)^j(n_2 n_1^i -1)jx.$$
  If $2\nmid i$, this simplifies to
  $$2x\sum_{j=0}^{m-1}(-1)^{j+1}j = mx.$$
  So for (\ref{gmeqn2})=1 to hold, we must have $2^l\mid mx$ in this case.  This gives vi).  On the other hand, if $2\mid i$, then we get
  \begin{eqnarray*}
    2^{l-2}mr + \sum_{j=0}^{m-1} n_2^j(n_2-1)jx &=& 2^{l-2}mr + (n_2-1)x \sum_{j=0}^{m-1}n_2^j j\\
    &\equiv& 2^{l-2}mr + (n_2-1)x\frac{m}{2}(2^{l-1}-1)\\
    &\equiv& 2^{l-2}mr - \frac{m}{2}x(n_1-1)\bmod 2^l.\\
  \end{eqnarray*}
  Thus we must have
  $$2^{l-2}mr \equiv mx(2^{l-2}+1)\bmod 2^l$$
  if (\ref{gmeqn2})=1 is to hold.  If $4\mid mr$, this is equivalent to $2^l\mid mx$.  Else, if $4\nmid mr$ (equivalently, $4\nmid m$ and $2\nmid r$), we have
  \begin{eqnarray*}
    2^{l-2}mr \equiv mx(2^{l-2}+1)\bmod 2^l &\iff& 2^{l-2}r\equiv (2^{l-2}+1)x \bmod 2^{l-1}\\
    &\iff& x\equiv 2^{l-2}(2^{l-2}+1)r \bmod 2^{l-1}\\
    &\iff& x\equiv 2^{l-2}\bmod 2^{l-1}.
  \end{eqnarray*}
  This gives iv) and completes the case of $2\mid y$.

  So suppose now that $2\nmid y$.  As mentioned before, this forces $4\mid m$.  Then the condition on the power on $a$ in (\ref{gmeqn2}) is given by
  $$\frac{m}{2}(1+n_2 n_1^{i+1})r + \sum_{j=0}^{m-1}(n_2 n_1^{i+1})^j (n_2 n_1^i-1) jx \equiv 0\bmod 2^l.$$
  If we have $2\mid i$, then is equivalent to
  \begin{eqnarray*}
    \sum_{j=0}^{m-1}(-1)^j (n_2-1)jx \equiv -mx(n_2-1)\equiv 2mx\equiv 0\bmod 2^l \iff 2^{l-1}\mid mx,
  \end{eqnarray*}
  where we have used Corollary \ref{basicident} in the second equivalence.  This gives v).  So finally suppose that $2\nmid i$, instead.  Then
  \begin{eqnarray*}
    \frac{m}{2}(1+n_2 n_1^{i+1})r + \sum_{j=0}^{m-1}(n_2 n_1^{i+1})^j (n_2 n_1^i-1) jx &\equiv& 2^{l-2}mr -2x \sum_{j=0}^{m-1}n_2^j j\\
    &\equiv& 2^{l-2}mr - mx(2^{l-2}-1)\bmod 2^l.
  \end{eqnarray*}
  Thus we must have
  $mx(2^{l-2}-1)\equiv 2^{l-2}mr \equiv 0 \bmod 2^l \iff 2^l\mid mx$.  This proves vii).

  \item[viii-xi)] By equation (\ref{genconj}) and Lemma \ref{gmsetlem}, we have
  \begin{eqnarray*}
    (a^x u^y v)^{-j} (a^s u^i v) (a^x u^y v)^j &=& a^{(n_2 n_1^y)^j\left( s+ (n_1^i n_2-1)\left( \ceil{j/2} + \floor{j/2}n_2 n_1^y\right)x\right)} u^{(-1)^j(i-2 \delta_{2\nmid j}y)} v.
  \end{eqnarray*}
  For notational convenience, set $N_j =  s+ (n_1^i n_2-1)\left( \ceil{j/2} + \floor{j/2}n_2 n_1^y\right)x$.  Then we find
  \begin{eqnarray}\label{gmeqn3}
    \prod_{j=0}^{m-1} (a^x u^y v)^{-j} (a^s u^i v) (a^x u^y v)^j &=& a^{\sum_{j=0}^{m-1} (n_1^{y+i})^j N_j} u^{m(i-y)} v^m.
  \end{eqnarray}
  Looking at the powers on $u$ and $v$, we see that if (\ref{gmeqn3})=1 were to hold, we must have $2\mid m \wedge k\mid m(i-y)$.  So suppose for the remainder of the proof that $2\mid m$ and $k\mid m(i-y)$.  Note that if $i\not\equiv y\bmod 2$, then $k\mid m(i-y)$ implies $4\mid m$.
  We now simplify the power on $a$ as follows
  \begin{eqnarray}
    \sum_{j=0}^{m-1} (n_1^{y+i})^j N_j &\equiv& (n_1^{y+i}+1)\frac{m}{2}s + n_1^{y+i}(n_1^i n_2-1)x \sum_{j=0}^{m-1}\delta_{2\nmid j}\nonumber\\
    &\equiv& (n_1^{y+i}+1)\frac{m}{2}s + n_1^y(n_2-n_1^i)\frac{m}{2}x\nonumber\\
    &\equiv& (n_1^{y+i}+1)\frac{m}{2}s + (n_2-n_1^i)\frac{m}{2}x,\label{gmeqn4}
  \end{eqnarray}
  with the last equivalence following from $2\mid(n_2-n_1^i)$ and Corollary \ref{basicident}.

  Breaking equation (\ref{gmeqn4}) down into cases depending on the parity of $i$ and $y$, as we did for parts (iii)-(vii), yields all remaining claims.
  \end{enumerate}
\end{proof}

\section{The Type I and II modules}\label{type1sect}
\subsection{Type I}
We now proceed to explicitly determine the indicators for Type I modules.  We start with the odd modules, and then finish with the two cases for the even modules.
\begin{thm}\label{oddtype1}
    Let $G=\G$ be as in Definition \ref{presentation}.  Consider the 4-dimensional irreducible $\D(G)$-module $V = M(\class(a^r u^{2i}),\phi_{s,t})$, with $2\nmid r$.  Let $1<m\in \BN$.  Then
    \begin{eqnarray*}
      \nu_m(V) = \left\{ \begin{array}{cll}
        4 &;& 2^l\mid m \wedge k\mid 2mi \wedge k\mid mt\\
        2 &;& 4\mid m \wedge (2^l\nmid m \vee k\nmid 2mi \vee k\nmid mt)\\
        1 &;& 2\mid m \wedge 4\nmid m\\
        0 &;& 2\nmid m \end{array}\right..
    \end{eqnarray*}
    In particular, $\nu_2(V)=1$.
\end{thm}
\begin{proof}
  Under the assumptions of the theorem, we have the following from Proposition \ref{gmsets}:
  \begin{eqnarray*}
    a^x u^y \in G_m(a^r u^{2i}) \iff& k\mid 2mi \wedge 2^l\mid mr &\iff k\mid 2mi \wedge 2^l\mid m\\
    a^x u^y v \in G_m(a^r u^{2i}) \iff& 2\mid m \wedge (2\nmid y \vee 4\mid mr) &\iff (2\mid m\wedge 2\nmid y) \vee 4\mid m.
  \end{eqnarray*}
  We compute the contribution to $\nu_m(V)$ from elements of the form $a^x u^y$ first.

  If $k\nmid 2mi \vee 2^l\nmid m$, then elements of the form $a^x u^y$ contribute 0 to the indicator.  So we may suppose for this portion of the proof that $k\mid 2mi \wedge 2^l\mid m$.  First we have
  \begin{eqnarray}\label{ot1eqn1}
    (a^x u^y)^m &=& a^{\left( \ceil{\frac{m}{2}} + \floor{\frac{m}{2}}n_1^j\right)x} u^{my}.
  \end{eqnarray}
  Now $C_G(a^r u^{2i}) = \cyc{a,u^2}\lhd G$ by Proposition \ref{centralizers}.  So by Corollary \ref{cor34} and Proposition \ref{classes} and its proof, for (\ref{ot1eqn1}) to give a non-zero contribution to $\nu_m(V)$ we must have $2\mid my$.  Since $2^l\mid m$, this always holds.

  So suppose that $2\mid y$.  Then the contribution from elements of the form $a^s u^j$ is given by
  \begin{eqnarray}\label{ot1eqn2}
    \frac{4}{|G|}\sum_{s=1}^{2^l}\sum_{2\mid y} \alpha^s(a^{mx}) \beta^t(u^{my}) &=& \ip{\phi_{ms,mt}}{1}.
  \end{eqnarray}
  Note that the $4$ ultimately corresponds to $|\class(a^r u^{2i})|$.  Each appropriate conjugate of $a^{mx} b^{my}$ that we would apply the character to in accordance with Corollary \ref{cor34} just changes the powers on $a$ and $b$ by some unit.  Since we are, in both cases, summing over all possible indices $s,j$, the sum is equivalent to the first one given in equation (\ref{ot1eqn2}).

  On the other hand, suppose that $2\nmid y$.  Suppressing unnecessary conjugates as before, the contribution from these elements is
  \begin{eqnarray*}
    \frac{4}{|G|}\sum_{x=1}^{2^l}\sum_{2\nmid y} \alpha^s(a^{\frac{m}{2}(n_1+1)x}) \beta^{\frac{m}{2}t}(u^{2y}) &=& \frac{4}{|G|}\sum_{x=1}^{2^l}\sum_{2\nmid y} \alpha^{ms}(a^x)\beta^{\frac{m}{2}y}(u^{2y})\\
    &=& \frac{4}{|G|}\sum_{x=1}^{2^l}\alpha^{ms}(a^x) \sum_{j=1}^{k/2} \beta^{\frac{m}{2}t}(u^{2(2j-1)})\\
  \end{eqnarray*}
  We can rewrite this last summation as a difference to get
  \begin{eqnarray}\label{ot1eqn3}
    \frac{4}{|G|}\sum_{x=1}^{2^l}\alpha^{ms}(a^x)\left( \sum_{j=1}^{k} \beta^{\frac{m}{2}t}(u^{2(2j-1)}) - \sum_{j=1}^{k/2} \beta^{\frac{m}{2}t}(u^{4j}) \right)\nonumber\\
    = 2\ip{\phi_{ms,\frac{m}{2}t}}{1} - \ip{\phi_{ms,mt}}{1}.
  \end{eqnarray}

  Combining, under the assumption that $k\mid 2mi \wedge 2^l\mid m$, the elements of the form $a^x u^y$ contribute the following the $\nu_m(V)$:
  \begin{eqnarray}\label{ot1part1}
    2\ip{\phi_{ms,mt}}{1}.
  \end{eqnarray}

  Now we compute the contribution from elements of the $a^x u^y v$.  If $4\nmid m \wedge (2\nmid m \vee 2\mid y)$, such an element contributes $0$ to the summation.  So assuming that $(2\mid m \wedge 2\nmid y)\vee 4\mid m$ holds, by Corollary \ref{basicident} we see that $(a^x u^y v)^m =1$.  When $2\mid m  \wedge 4\nmid m$, there are $2^{l-1}k$ elements that satisfy the necessary condition.  And when $4\mid m$, every $x$ and $y$ work, yielding $2^l k$ such elements.  Thus the elements of the form $a^x u^y v$ contribute the following to $\nu_m(V)$:
  \begin{eqnarray}\label{ot1part2}
    \delta_{2\mid m} + \delta_{4\mid m}.
  \end{eqnarray}

  Combining everything together, we get the desired formula.  That $\nu_2(V)=1$ is then immediate.
\end{proof}

Similar arguments to those given above show that, in all of our remaining computations in this and other sections, we may suppress the units that arise from taking the appropriate conjugates when computing the value of the sum in Corollary \ref{cor34}.  We shall do so without further note, though we point out now that Lemma \ref{type3lem} is essential for justifying this approach for the Type III modules.

\begin{thm}\label{even1dimtype1}
Let $G=\G$ be as in Definition \ref{presentation}.  Suppose $2\mid r$ and $a^r u^{2i}\not\int Z(G)$, and let $\psi_{s,t}\alpha^s\otimes \beta^t$ be an irreducible linear character of $C_G(a^r u^{2i})$.  Consider the 2-dimensional irreducible $\D(G)$-module $V = M(\class(a^r u^{2i}),\psi_{s,t})$.  Let $1<m\in \BN$, and let $P$ be the proposition $P=k\mid 2mi \wedge 2^l\mid mr \wedge \psi_{ms,mt}=1$.  Then
\begin{eqnarray*}
  \nu_m(V) &=& \delta_{2\mid m} + \delta_P\\
  &=& \left\{ \begin{array}{cll}
    2 &;& 2\mid m \wedge P\\
    1 &;& 2\mid m \oplus P\\
    0 &;& 2\nmid m \wedge \neg P
    \end{array}\right..
\end{eqnarray*}
In particular, $\nu_2(V) = 1$.
\end{thm}
\begin{proof}
  Under the assumptions of the theorem, we have the following from Proposition \ref{gmsets}:
  \begin{eqnarray*}
    a^x u^y \in G_m(a^r u^{2i}) \iff& k\mid 2mi \wedge 2^l\mid mr &\iff k\mid 2mi \wedge 2^l\mid m\\
    a^x u^y v \in G_m(a^r u^{2i}) \iff& 2\mid m \wedge (2\nmid y \vee 4\mid mr) &\iff 2\mid m.
  \end{eqnarray*}
  We compute the contribution to $\nu_m(V)$ from elements of the form $a^x u^y$ first.

  Without loss of generality, when computing the contribution of elements of the form $a^x u^y$ we may suppose that $k\mid 2mi \wedge 2^l\mid mr$.

  Similar to the proof of the previous theorem, we find that the contribution from those terms with $2\mid y$ is given by
  \begin{eqnarray}\label{et1eqn1}
    \frac{2}{|G|}\sum_{x=1}^{2^l}\sum_{2\mid y} \alpha^{ms}(a^x)\beta^{mt}(u^y).
  \end{eqnarray}
  Note that in this case we do not express this an inner product of characters of $\cyc{a,u^2}$.  Since our full centralizer is $\cyc{a,u}$ instead of $\cyc{a,u^2}$ this time, we will find it easier to combine this formula with that for the $2\nmid y$ case than in the previous proof.

  On the other hand, suppose now that $2\nmid y$.  It may or may not happen that $2\mid m$ in this situation, so we break the computation into two further subcases.  If $2\mid m$, the contribution from these elements to the indicator is
  \begin{eqnarray}\label{et1eqn2}
    \frac{2}{|G|}\sum_{x=1}^{2^l}\sum_{2\nmid y} \alpha^{ms}(a^{(2^{l-2}+1)x}) \beta^{mt}(u^y) &=& \frac{2}{|G|}\sum_{x=1}^{2^l} \alpha^{ms}(a^x)\beta^{mt}(u^y).
  \end{eqnarray}
  Otherwise, if $2\nmid m$, the contribution is
  \begin{eqnarray}\label{et1eqn3}
    \frac{2}{|G|}\sum_{x=1}^{2^l}\sum_{2\nmid y} \alpha^s(a^{(1+(m-1)(2^{l-2}+1))x})\beta^{mt}(u^y) &=& \frac{2}{|G|}\sum_{x=1}^{2^l} \alpha^{ms}(a^x)\beta^{mt}(u^y),
  \end{eqnarray}
  where we have used the fact that $2\nmid m$ to conclude that $1+(m-1)(2^{l-2}+1)$ is a unit modulo $2^l$.

  Combining these together, we find that elements of the form $a^x u^y$ contribute $\delta_P$ to $\nu_m(V)$.

  We now consider the contribution from elements of the form $a^x u^y v$.  We have shown that we may suppose that $2\mid m$ if these elements are to contribute a non-zero value to $\nu_m(V)$.  By Corollary \ref{basicident}, each such element contributes a value of either $1$ or $\psi_{s,t}(a^{2^{l-1}})$, depending on the parities of $x,y$, and $m/2$.  Since $\psi_{s,t}$ is a linear character of $\cyc{a,u}$, however, we necessarily have $\psi_{s,t}(a^{2^{l-1}}) = 1$.  Since there are $2^l\cdot k$ elements of the form $a^x u^y v$, the contribution of these to $\nu_m(V)$ is exactly $\delta_{2\mid m}$.

  Combining both parts together, we get $\nu_m(V) = \delta_{2\mid m} + \delta_P$ as desired.

  For the particular case of $m=2$, we get $\nu_2(V)=1 + \delta_P$.  But $\delta_P=1$ forces $2^{l-1}\mid r \wedge k\mid 4i$, which is equivalent to $a^r u^{2i}\in Z(G)$.  Thus $\nu_2(V)=1$ as claimed.
\end{proof}

\begin{thm}\label{even2dimtype1}
Let $G=\G$ be as in Definition \ref{presentation}.  Suppose $2\mid i$, and let $\phi_{r,s}$ be an irreducible 2-dimensional character of $C_G(a^i u^{2j})$, with $\phi_{r,s}$ defined as in \cite[Lemma 3.1]{K}.  Consider the 4-dimensional irreducible $\D(G)$-module $V = M(\class(a^i u^{2j}), \phi_{r,s})$.  Let $1<m\in \BN$.

Then
$$\nu_m(V)= \left\{ \begin{array}{cll}
    4 &;& 2^l\mid m \wedge k\mid mt \wedge k\mid 2mi\\
    2 &;& 4\mid m \wedge (2^l\nmid m \vee k\nmid mt \vee k\nmid 2mi)\\
    1 &;& 2\mid m \wedge 4\nmid m\\
    0 &;& 2\nmid m \end{array}\right..$$
In particular, $\nu_2(V)=1$.
\end{thm}
\begin{proof}
    We proceed exactly as in the proofs of Theorems \ref{oddtype1} and \ref{even1dimtype1}.  Similarly as in the proof of Theorem \ref{even1dimtype1}, we find that the contribution from elements of the form $a^x u^y$ is 0 if $k\nmid 2mi \vee 2^l\nmid mr$, and otherwise they contribute $\ip{\phi_{mr,ms}}{1}$.  Since $\phi_{r,s}$ is induced from a linear character of $\cyc{a,u^2}$, and we must have $2\nmid r$, we conclude that $\ip{\phi_{mr,ms}}{1}$ is either 2 or 0, and that it is $2$ precisely when $2^l\mid m \wedge k\mid mt$.

    For the contribution for the $a^x u^y v$ terms, we again conclude that we must have $2\mid m$ for a non-zero contribution.  By \cite[Lemma 3.1]{K}, we have $2\nmid r$, and so $\phi_{r,s}(a^{2^{l-1}}) = -2$.  Then by Corollary \ref{basicident}, the contribution from terms of the form $a^x u^y v$ is given by
    \begin{eqnarray*}
      \frac{2}{|G|} \left\{ \begin{array}{cll}
        2^{l+1} k &;& 4\mid m\\
        3\cdot 2^{l-1} k - 2^{l-1}k &;& 2\mid m \wedge 4\nmid m
      \end{array}\right. &=& \left\{ \begin{array}{cll}
        2 &;& 4\mid m\\
        1 &;& 2\mid m \wedge 4\nmid m \end{array}\right..
    \end{eqnarray*}

    Since $2^l\mid m$ forces $4\mid m$ in particular, we conclude that $\nu_m(V)$ takes the stated values under the stated conditions.  That $\nu_2(V)=1$ is then immediate.
\end{proof}

\subsection{Type II}\label{type2sect}

We now consider the indicators of the Type II modules.  We start with a simple intermediate lemma, which lets us deal with the varying centralizers for Type II modules simultaneously (see Proposition \ref{centralizers}.ii).  The proofs for the results here are again by direct computation, and obtained in much the same fashion as for the Type I modules.  As such, we omit them here.

\begin{lem}\label{type2lem}
  Suppose $G=\G$ is a finite group as in Definition \ref{presentation}.  Let $g\in G_m(a^r u^{2i-1})$.  Then $g^m\in \cyc{a^2,u^2}$, an abelian normal subgroup of $G$ isomorphic to $\BZ_{2^{l-1}}\times \BZ_{k/2}$.
\end{lem}

\begin{thm}\label{type2}
  Let $G=\G$ be as in Definition \ref{presentation}.  Let $\eta$ be any irreducible linear character of $C_G(a^r u^{2i-1})$.  Set $H=\cyc{a^2, u^2}\subseteq C_G(a^r u^{2i-1})$.  Consider the 4-dimensional irreducible $\D(G)$-module $V=M(\class(a^r u^{2i-1}),\eta)$.  Then for any $1<m\in\BN$ we have
  $$\nu_m(V) = \left\{ \begin{array}{cll}
    4 &;& k\mid m(2i-1) \wedge 2^l\mid mr \wedge \eta^{m/2}|_H = 1\\
    2 &;& 4\mid m \wedge (k\nmid m(2i-1) \vee 2^l\nmid mr \vee \eta^{m/2}|_H\neq 1)\\
    1 &;& 4\nmid m \wedge 2\mid m\\
    0 &;& 2\nmid m
  \end{array}\right.$$
  In particular, $\nu_2(V)=1$.
\end{thm}

\section{The Type III modules}\label{type3sect}
We finish with the indicators of the Type III modules.  We start with a simple intermediate lemma, which helps us deal with the varying centralizers for Type III modules (see Proposition \ref{centralizers}.iii), similar to our use for Lemma \ref{type2lem}.

\begin{lem}\label{type3lem}
    Suppose $G=\G$ is a finite group as in Definition \ref{presentation}.  Let $g\in G_m(a^r u^{i} v)$.  Then $g^m\in \cyc{a^{2^{l-1}}}\subset Z(G)$.
\end{lem}
\begin{proof}
  Lengthy but routine computations of $g^m$ under the conditions that $g\in G_m(a^r u^i v)$ gives the desired result.
\end{proof}

We will find it convenient to break down the Type III modules along parity lines, in accordance with Proposition \ref{classes}.

\begin{thm}\label{oddtype3}
    Let $G=\G$ be as in Definition \ref{presentation}.  Let $\eta$ be any irreducible linear character of $C_G(a^r u^{i} v)$, and suppose $2\nmid i$.  Consider the $2^{l-2}k$-dimensional irreducible $\D(G)$-module $V=M(\class(a^r u^{i}v),\eta)$.  Then
    $$\nu_m(V) = \delta_{2\mid m}\frac{1}{4}\gcd(m,k)\gcd(m,2^l).$$
In particular, $\nu_2(V) = 1$.
\end{thm}
\begin{proof}
  By Proposition \ref{gmsets}, we have
  \begin{eqnarray}\label{ot3gm1}
    a^x u^{2y} \in G_m(a^r u^i v) &\iff& 2\mid m \wedge k\mid 2my \wedge 2^l\mid mx\\
    \label{ot3gm2}a^x u^{2y-1}\in G_m(a^r u^i v)&\iff& 4\mid m \wedge k\mid m(2y-1) \wedge 2^l\mid mx\\
    \label{ot3gm3}a^x u^y v\in G_m(a^r u^i v) &\iff& 2\mid m \wedge k\mid m(y-i) \wedge 2^l\mid m(x-r).
  \end{eqnarray}
  So it follows that if $2\nmid m$ then $\nu_m(V)=0$.  So for the remainder of the proof we assume that $2\mid m$.

  We begin by determining the contribution from elements of the form $a^x u^y$.  We have
  $$(a^x u^y)^m = a^{\frac{m}{2}(n_1^y+1)x} u^{my}.$$
  Now $2\mid (n_1^y+1)$, and so by equations (\ref{ot3gm1}) and (\ref{ot3gm2}), we conclude that $(a^x u^y)^m =1$.  Furthermore, these equations tell us that there are $\gcd(m,k)$ values of $y$ and $\gcd(m,2^l)$ values of $x$ satisfying the necessary conditions.  Subsequently, elements of the form $a^x u^y$ contribute a value of
  \begin{eqnarray}\label{ot3part1}
    \delta_{2\mid m}\frac{1}{8} \gcd(m,k)\gcd(m,2^l)
  \end{eqnarray}
  to $\nu_m(V)$.

  We now consider the contribution from elements of the form $a^x u^y v$.  By Proposition \ref{basicident}, $(a^x u^v v)^m$ is either $1$ or $a^{2^{l-1}}$.  For the latter to hold, we must have $2\mid y$.  Since also $k\mid m(y-i)$ in this case by (\ref{ot3gm3}), we conclude that $4\mid m$.  Therefore, $(a^x u^y v)^m =1$ whenever $a^x u^y v\in G_m(a^r u^i v)$.  By (\ref{ot3gm3}), there are clearly $\gcd(m,k)$ such values of $y$ and $\gcd(m,2^l)$ such values of $x$.  Therefore, these elements contribute
  \begin{eqnarray}\label{ot3part2}
    \delta_{2\mid m}\frac{1}{8} \gcd(m,k)\gcd(m,2^l)
  \end{eqnarray}
  to $\nu_m(V)$.  Combining this with (\ref{ot3part1}) gives the desired formula.  That $\nu_2(V)=1$ is then immediate.
\end{proof}

\begin{thm}\label{eveneventype3}
  Let $G=\G$ be as in Definition \ref{presentation}.  Let $\eta$ be any irreducible linear character of $C_G(a^r u^{i} v)$, and suppose $2\mid i$ and $2\mid r$.  Consider the $2^{l-2}k$-dimensional irreducible $\D(G)$-module $V=M(\class(a^r u^{i}v),\eta)$.
  \begin{enumerate}
    \item If $2\nmid m$, $\nu_m(V) = 0$.
    \item If $2\mid m$ and $2\gcd(m,k)\mid k$, then
    $$\nu_m(V) = \frac{1}{4}\gcd(m,k)\gcd(m,2^l).$$
    \item If $2\mid m \wedge 2\gcd(m,k)\nmid k \wedge \eta(a^{2^{l-1}}) = 1$, then
    $$\nu_m(V) = \frac{1}{16}\gcd(m,k)\left(3\gcd(m,2^l) + 2\gcd(m,2^{l-1})\right).$$
    \item If $2\mid m \wedge 2\gcd(m,k)\nmid k \wedge \eta(a^{2^{l-1}}) = -1$, then
    $$\nu_m(V) = \frac{1}{16}\gcd(m,k)\left(5\gcd(m,2^l) - 2\gcd(m,2^{l-1})\right).$$
  \end{enumerate}
In particular, $\nu_2(V) = 1$.
\end{thm}
\begin{proof}
  Set $g=a^r u^i v$.  By Proposition \ref{gmsets}, we have
  \begin{eqnarray}\label{eet3gm1}
    a^x u^{2y} \in G_m(g) &\iff& 2\mid m \wedge k\mid 2my \wedge 2^l\mid mx\\
    \label{eet3gm2}a^x u^{2y-1}\in G_m(g)&\iff& 4\mid m \wedge k\mid m(2y-1) \wedge 2^{l-1}\mid mx\\
    \label{eet3gm3}a^x u^{2y} v\in G_m(g) &\iff& 2\mid m \wedge k\mid m(2y-i) \wedge 2^l \mid m(r-(1-2^{l-2})x)\\
    \label{eet3gm4}a^x u^{2y-1} v \in G_m(g) &\iff& 4\mid m \wedge k\mid m(2y-1-i) \wedge 2^l\mid m(r-x)
  \end{eqnarray}
  So it follows that if $2\nmid m$ then $\nu_m(V)=0$.  So for the remainder of the proof we assume that $2\mid m$.

  We begin by determining the contribution from elements of the form $a^x u^y$.  For any $a^x u^y \in G_m(a^r u^i v)$ we have
  $$(a^x u^y)^m = a^{\frac{m}{2}(n_1^y+1)x} u^{my} = \left\{ \begin{array}{cll}
    1 &;& 2\mid y \vee 2^l \mid mx\\
    a^{2^{l-1}} &;& 2\nmid y \wedge 2^l\nmid mx \end{array}\right..$$
  Since for any such $y$ we must have $k\mid my$, it follows that there are $\gcd(m,k)$ possible values for $y$.  If $k/\gcd(m,k)$ is even, then all such values are even, and so $(a^x u^y)^m=1$.  On the other hand, if $k/\gcd(m,k)$ is odd, then half of the valid values for $y$ are even and the other half are odd.  Note that $k/\gcd(m,k)$ odd forces $4\mid m$.  Supposing that $y$ is odd, we readily find that there are $2\gcd(m,2^{l-1})$ possible values for $x$, and that $\gcd(m,2^l)$ of them have $2^l\mid mx$.  For convenience, let $P$ be the proposition $2\mid m \wedge 2\gcd(m,k)\mid k$ and let $Q$ be the proposition $4\mid m \wedge 2\gcd(m,k)\nmid k$.  All combined, we conclude that elements of the form $a^x u^y$ contribute
  \begin{eqnarray*}
    \left\{ \begin{array}{cll}
        \frac{1}{8}\gcd(m,k)\gcd(m,2^l) &;& P\\
        \frac{1}{16}\gcd(m,k)\left( 2\gcd(m,2^l) + \eta(a^{2^{l-1}})(2\gcd(m,2^{l-1})-\gcd(m,2^l))\right) &;& Q\\
        0&;& 2\nmid m
    \end{array}\right.
  \end{eqnarray*}
  to $\nu_m(V)$.

  For the contribution from elements of the form $a^x u^y v$, we give an argument similar to that in the proof of Theorem \ref{oddtype3} to conclude that we once again get a contribution equal to (\ref{ot3part2}).  The desired formula follows, and $\nu_2(V)=1$ is then immediate.
\end{proof}

We now have a single case of modules left to consider to complete our goal.

\begin{thm}\label{oddeventype3}
  Let $G=\G$ be as in Definition \ref{presentation}.  Let $\eta$ be any irreducible linear character of $C_G(a^r u^{i} v)$, and suppose $2\mid i$ and $2\nmid r$.  Consider the $2^{l-2}k$-dimensional irreducible $\D(G)$-module $V=M(\class(a^r u^{i}v),\eta)$.  Then
  \begin{enumerate}
    \item If $2\nmid m$, then $\nu_m(V) = 0$.
    \item If $2\mid m \wedge 4\nmid m$, then
    $$\nu_m(V) = \frac{1}{2}\gcd(m,k) \eta(a^{2^{l-1}}).$$
    \item If $4\mid m \wedge 2\gcd(m,k)\mid k$, then
    $$\nu_m(V) = \frac{1}{4}\gcd(m,k)\gcd(m,2^l).$$
    \item If $4\mid m \wedge 2\gcd(m,k)\nmid k \wedge \eta(a^{2^{l-1}}) = 1$, then
    $$\nu_m(V) = \frac{1}{16}\gcd(m,k)\left( 3\gcd(m,2^l) + 2 \gcd(m,2^{l-1}) \right).$$
    \item If $4\mid m \wedge 2\gcd(m,k)\nmid k \wedge \eta(a^{2^{l-1}}) = -1$, then
    $$\nu_m(V) = \frac{1}{16}\gcd(m,k)\left( 5\gcd(m,2^l) - 2 \gcd(m,2^{l-1}) \right).$$
  \end{enumerate}
  In particular, $\nu_2(V) = \eta(a^{2^{l-1}})=\pm 1$.
\end{thm}
\begin{proof}
  Set $g=a^r u^i v$.  By Proposition \ref{gmsets}, we have
  \begin{eqnarray*}
    a^x u^{2y} \in G_m(g) &\iff& 2\mid m \wedge k\mid 2my \wedge \ \\
     && \ \left((4\mid m \wedge 2^l \mid mx) \vee (4\nmid m \wedge x\equiv 2^{l-2}\bmod 2^{l-1})\right)\\
    a^x u^{2y-1}\in G_m(g)&\iff& 4\mid m \wedge k\mid m(2y-1) \wedge 2^{l-1}\mid mx\\
    a^x u^{2y} v\in G_m(g) &\iff& 2\mid m \wedge k\mid m(2y-i) \wedge 2^l \mid m(r-(1-2^{l-2})x)\\
    a^x u^{2y-1} v \in G_m(g) &\iff& 4\mid m \wedge k\mid m(2y-1-i) \wedge 2^l\mid m(r-x)
  \end{eqnarray*}
  So it follows that if $2\nmid m$ then $\nu_m(V)=0$.  So for the remainder of the proof we assume that $2\mid m$.

  We begin by determining the contribution from elements of the form $a^x u^y$.  For any $a^x u^y \in G_m(a^r u^i v)$ we have
  \begin{eqnarray*}
    (a^x u^y)^m &=& \left\{ \begin{array}{cll}
        a^{mx} &;& 2\mid y\\
        a^{(2^{l-2}+1)mx} &;& 2\nmid y \end{array}\right..
  \end{eqnarray*}
  If $2\mid y$, then
  \begin{eqnarray*}
    (a^x u^y)^m &=& \left\{ \begin{array}{cll}
        1 &;& 4\mid m \wedge 2^l\mid mx\\
        a^{2^{l-1}} &;& 4\nmid m \wedge x\equiv 2^{l-2}\bmod 2^{l-1}\end{array}\right..
  \end{eqnarray*}
  There are $\gcd(m,2^l)$ $x$ satisfying the first part, and exactly $2$ values satisfying the second.  On the other hand, if $2\nmid y$, then
  \begin{eqnarray*}
    (a^x u^y)^m &=& \left\{ \begin{array}{cll}
        1 &;& 4\mid m \wedge 2^l\mid mx\\
        a^{2^{l-1}} &;& 4\mid m \wedge 2^{l-1}\mid mx \wedge 2^l\nmid mx\end{array}\right..
  \end{eqnarray*}
  There are $\gcd(m,2^l)$ values for $x$ satisfying the first part, and \newline $2\gcd(m,2^{l-1}) - \gcd(m,2^l)$ satisfying the second.  Any valid value for $y$ mlust satisfy $my\equiv mi\bmod k$, and there is a total $\gcd(m,k)$ such (distinct) values.  Furthermore, when $k/\gcd(m,k)$ is even then every valid value for $y$ is even, and when $k/\gcd(m,k)$ is odd then half of them are odd and half are even.  Let $P$ be the proposition $2\gcd(m,k)\mid k$.  Considering all cases, the contribution to $\nu_m(V)$ from elements of the form $a^x u^y$ is given by
  \renewcommand\arraystretch{2.25}
  $$\left\{\begin{array}{cll}
    \ds{\frac{1}{4}\gcd(m,k)\eta(a^{2^{l-1}})} &;& 4\nmid m\\
    \ds{\frac{1}{8} \gcd(m,k)\gcd(m,2^l)} &;& 4\mid m \wedge P\\
    \ds{\frac{1}{16} \gcd(m,k) \left( 2\gcd(m,2^l) + \eta(a^{2^{l-1}})( 2\gcd(m,2^{l-1}) - \gcd(m,2^l)) \right)} &;& 4\mid m \wedge \neg P.
  \end{array}\right.$$

  For the contribution from $a^x u^y v$ we proceed in a similar fashion to the proof of Theorem \ref{oddtype3} and get
  \renewcommand\arraystretch{2.25}
  $$\left\{ \begin{array}{cll}
    \ds{\frac{1}{8} \gcd(m,k)\gcd(m,2^l)} &;& 4\mid m\\
    \ds{\frac{1}{4} \gcd(m,k) \eta(a^{2^{l-1}})} &;& 2\mid m \wedge 4\nmid m\\
    0 &;& 2\nmid m
    \end{array}\right..$$

    Combining things together, we get the desired formula for $\nu_m(V)$.  That $\nu_2(V) = \eta(a^{2^{l-1}})$ then follows.
\end{proof}

We thus get our main theorem as an easy corollary.
\begin{thm}\label{doublesum}
There exist infinitely many finite groups $G$ having all of the following properties
\begin{enumerate}
  \item $G$ is completely real.
  \item $G$ is generated by involutions.
  \item $\forall g\in G$ $C_G(g)$ is an $M$-group.  In particular, $G$ is an $M$-group.
  \item $G$ is totally orthogonal.
  \item $\nu_m(V)\geq 0$ for every $G$-module $V$.
  \item Every (irreducible) $\D(G)$-module is self-dual.
  \item There exists an irreducible $\D(G)$-module $W$ such that $\nu_2(W)=-1$.
\end{enumerate}
\end{thm}
\begin{proof}
  We can take any group $\G$ from Definition \ref{presentation}.  Parts i-v) are the content of Theorem \ref{groupsummary}.  Part vi) follows from: Lemma \ref{FxF} applied to Theorem \ref{groupsummary} and Corollary \ref{center}; and the Theorems of Sections \ref{type1sect}-\ref{type3sect}.

  Now $C_G(a^r u^i v)$ is abelian by Proposition \ref{centralizers}.  Therefore, we can always arrange for one of its characters $\eta$ to have $\eta(a^{2^{l-1}}) = -1$.  Thus part vii) holds by Theorem \ref{oddeventype3}.
\end{proof}
In \cite{GM} it was asked if the indicators for $\D(G)$ are non-negative whenever $G$ is totally orthogonal. We can now assert that this is not the case, even if we restrict ourselves to just the classical indicators $\nu_2(V)$.  We remark again that D. Naidu has independently shown that $\BZ_8\rtimes D_4$, in the notation of Definition \ref{presentation}, was a counterexample.  It has been conjectured that non-negativity of the indicators for $\D(G)$ holds provided we further require that $G$ be a real reflection group.  Such groups are well-known to satisfy conditions i), ii), and iv) of our theorem (the fifth part remains unproven, except for the dihedral \cite{K} and symmetric \cite{Sch} groups), and in \cite{GM} it was shown that $\D(G)$ is also totally orthogonal.  Theorem \ref{doublesum} shows that the condition of being generated by reflections in the conjecture cannot be relaxed to simply "generated by involutions".  Furthermore, by parts iv), vi), and vii) we can now positively answer the existence part of Question 4 in \cite{K}.

\bibliographystyle{plain}
\bibliography{references}

\end{document}